\documentclass[12pt]{article}
\usepackage{amsmath, amssymb, amsfonts}
\usepackage{amsthm}
\usepackage[utf8]{inputenc}
\usepackage{csquotes}
\usepackage{geometry}
\usepackage{adjustbox}
\usepackage{picture}
\usepackage{graphicx}
\usepackage{verbatim}
\usepackage{array}
\usepackage{enumitem}
\usepackage{booktabs}
\usepackage{hyperref}
\usepackage{url}
\usepackage[capitalize]{cleveref}
\usepackage{tikz}

\title{Fractal Origin of the Continuum: A Hypothesis on Process-Relative Definability}
\author{Stanislav Semenov \\
\href{mailto:stas.semenov@gmail.com}{stas.semenov@gmail.com} \\
\href{https://orcid.org/0000-0002-5891-8119}{ORCID: 0000-0002-5891-8119}}
\date{April 5, 2025}

\theoremstyle{definition}
\newtheorem{definition}{Definition}[section]

\theoremstyle{plain}

\newtheorem{theorem}[definition]{Theorem}

\newtheorem{corollary}[definition]{Corollary}
\newtheorem{proposition}[definition]{Proposition}

\newtheorem{hypothesis}[definition]{Hypothesis}

\theoremstyle{remark}
\newtheorem*{remark}{Remark}

\begin{document}

\maketitle

\begin{abstract}
We propose a new constructive model of the real continuum based on the notion of \emph{fractal definability}. Rather than assuming the continuum as a completed uncountable totality, we view it as the cumulative result of a vast space of \emph{stratified formal systems}, each defining a countable layer of real numbers via constructive means. The union of all such definable layers across all admissible chains yields a set of continuum cardinality, yet no single system or definability path suffices to capture it in full. This leads to the \emph{Fractal Origin Hypothesis}: the apparent uncountability of the real line arises not from actual infinity, but from the meta-theoretical continuity of definability itself. Our framework models the continuum as a \emph{process-relative totality}, grounded in syntax and layered formal growth. We develop this idea through a formal analysis of \emph{definability hierarchies} and show that the resulting universe of constructible reals is \emph{countable-by-construction}\footnote{That is, each element is definable within some finite syntactic system, but no single procedure enumerates all of them uniformly.} yet inaccessible to any uniform enumeration. The continuum, in this view, is not a static set but a \emph{stratified semantic horizon}.
\end{abstract}

\subsection*{Mathematics Subject Classification}
03F60 (Constructive and recursive analysis), 26E40 (Constructive analysis), 03F03 (Proof theory and constructive mathematics)

\subsection*{ACM Classification}
F.4.1 Mathematical Logic, F.1.1 Models of Computation

\section*{Introduction}

\begin{center}
\textit{The continuum is not a set — it is the space-time of constructive definability.}
\end{center}

What is the origin of the continuum?

In classical set theory, the real line \( \mathbb{R} \) is introduced axiomatically as a completed, uncountable totality \cite{Cantor1883,Simpson2009}. Its cardinality is postulated to be \( 2^{\aleph_0} \), typically constructed via the power set of \( \mathbb{N} \), Dedekind cuts, or equivalence classes of Cauchy sequences. These constructions, however, assume access to a global infinity not available within constructive or syntactic frameworks \cite{Bishop1967,BridgesRichman1987}.

From a constructive standpoint, any formal system \( \mathcal{F} \) governed by enumerable rules can define only a countable subset of the reals—those that admit finite descriptions or convergent approximations in the language of \( \mathcal{F} \) \cite{Weihrauch2000,PourElRichards1989}. Consequently, within any such system, the continuum appears fragmented: a countable island in an uncountable sea.

This observation leads us to a fundamental question: could the apparent uncountability of \( \mathbb{R} \) arise not from ontological commitment to actual infinity, but from the \emph{meta-theoretical diversity of definability} \cite{Semenov2025FractalBoundaries}?

In this paper, we propose the \emph{Fractal Origin Hypothesis}: the continuum emerges as a syntactic horizon, a union of definability layers constructed along all admissible directions of formal growth. Each stratified chain of constructive systems \( \{ \mathcal{F}_n \} \) defines a countable, internally coherent model \( \mathbb{R}_{S_\omega}^{\{\mathcal{F}_n\}} \), but the class of all such chains—denoted \( \mathbb{F}_\omega \)—spans a space of cardinality \( \mathfrak{c} \). Taking the union over all such directions yields a model of the continuum that is \emph{countable-by-construction} yet \emph{not enumerable} within any single formal system.

This framework shifts the foundation of the continuum from external set-theoretic completion to internal syntactic growth \cite{Semenov2025FractalCountability,Semenov2025FractalAnalysis}. It allows us to model \( \mathbb{R} \) as a \emph{process-relative} totality—assembled fractally from the layered expressivity of formal systems. In doing so, we reframe the continuum as a semantic limit of definability itself.

\section{Constructive Systems and Extension Chains}

\begin{definition}[Constructive System]
A formal system \( \mathcal{F} \) is called \emph{constructive} if:
\begin{enumerate}
    \item It has a countable syntax;
    \item All inference and construction rules are syntactically enumerable;
    \item All definable objects can be represented as finite derivations or algorithms.
\end{enumerate}
\end{definition}

\begin{definition}[Constructive Extension Chain]
A sequence \( \{S_n\} \) of subsets of \( \mathbb{R} \) is a \emph{constructive extension chain} in system \( \mathcal{F} \) if:
\begin{itemize}
    \item \( S_0 \) is definable in \( \mathcal{F} \);
    \item For each \( n \), the set \( S_{n+1} \) is generated from \( S_n \) by a syntactically definable extension operator in \( \mathcal{F} \);
    \item Each \( S_n \subseteq \mathbb{R} \) is countable.
\end{itemize}
\end{definition}

\begin{definition}[Fractal Boundary \( S^{\mathcal{F}} \)]
The \emph{fractal boundary of constructivity} in \( \mathcal{F} \) is defined as the union over all such chains:
\[
S^{\mathcal{F}} := \bigcup_{\{S_n\}} \left( \bigcup_{n=0}^\infty S_n \right)
\]
where the outer union ranges over all constructive extension chains definable in \( \mathcal{F} \).
\end{definition}

\begin{theorem}[Fractal Closure is Countable]
\label{thm:closure}
Let \( \mathcal{F} \) be any constructive formal system. Then \( S^{\mathcal{F}} \subseteq \mathbb{R} \) is countable.
\end{theorem}

\begin{proof}
There are countably many constructive chains \( \{S_n\} \) definable in \( \mathcal{F} \), and each chain produces a countable union of sets. Hence \( S^{\mathcal{F}} \) is a countable union of countable sets.
\end{proof}

\section{Enumerability and Meta-Theoretical Limitations}

\begin{definition}[Enumerability in \( \mathcal{F} \)]
A set \( A \subseteq \mathbb{R} \) is \emph{enumerable in \( \mathcal{F} \)} if there exists a definable function \( e: \mathbb{N} \to A \) within \( \mathcal{F} \) that enumerates all elements of \( A \) without repetition.
\end{definition}

\begin{theorem}[Non-enumerability of \( S^{\mathcal{F}} \)]
\label{thm:nonenum}
Let \( \mathcal{F} \) be a constructive formal system. Then \( S^{\mathcal{F}} \) is not enumerable within \( \mathcal{F} \).
\end{theorem}

\begin{proof}
Suppose, for contradiction, that there exists a definable total function \( e \in \mathcal{F} \) such that \( e(n) \in S^{\mathcal{F}} \) for all \( n \in \mathbb{N} \), and the image of \( e \) covers all of \( S^{\mathcal{F}} \). Then all constructive extension chains definable in \( \mathcal{F} \) would be effectively encoded within a single enumeration procedure.

However, such a function would require \( \mathcal{F} \) to define the totality of its own definable sets — i.e., to enumerate all syntactically constructible functions and extensions expressible in \( \mathcal{F} \) itself. This contradicts general results in logic, including:

\begin{itemize}
    \item Kleene’s Recursion Theorem \cite{SoareRecursion}: which shows that no effectively enumerable system can enumerate all total computable functions definable within it.
    \item Gödel’s Second Incompleteness Theorem: no sufficiently expressive system can prove the totality (or consistency) of its own syntactic universe \cite{Simpson2009}.
\end{itemize}

Therefore, \( \mathcal{F} \) cannot define a function \( e \) enumerating all of \( S^{\mathcal{F}} \), and so \( S^{\mathcal{F}} \) is not enumerable within \( \mathcal{F} \).
\end{proof}

\begin{corollary}
There is no single constructive procedure inside \( \mathcal{F} \) that generates all elements of \( S^{\mathcal{F}} \).
\end{corollary}

\section{Ultimate Constructive Closure and Fractal Union}

In light of both internal and stratified approaches to definability, we consider a maximal constructive closure obtained by combining all possible internal chains across all countable formal systems.

\begin{definition}[Fractal Super-System]
Let \( \mathcal{F}_\omega \) denote the (meta-theoretical) union of all countable constructive formal systems. We define:
\[
S^{\mathcal{F}_\omega} := \bigcup_{\mathcal{F} \in \text{Constructive}} S^{\mathcal{F}},
\]
where each \( S^{\mathcal{F}} \) is the internal fractal boundary defined within system \( \mathcal{F} \).
\end{definition}

\begin{theorem}
The set \( S^{\mathcal{F}_\omega} \subseteq \mathbb{R} \) is countable, but no constructive system, including \( \mathcal{F}_\omega \), can uniformly enumerate all of its elements.
\end{theorem}

\begin{proof}
Each internal boundary \( S^{\mathcal{F}} \) is countable, so the union \( S^{\mathcal{F}_\omega} \) is also countable. However, no single system can define or enumerate all syntactic chains arising across all systems simultaneously. Thus, by a meta-extension of Theorem~\ref{thm:nonenum}, \( S^{\mathcal{F}_\omega} \) is not constructively enumerable.
\end{proof}

\begin{remark}
This set \( S^{\mathcal{F}_\omega} \) represents a meta-constructive horizon: a union of all internally reachable definable fragments. It serves as a constructive analog of the classical power set \( \mathcal{P}(\mathbb{N}) \): countable in cardinality, but unreachable by any uniform constructive process.
\end{remark}

\section{Two Modes of Stratified Definability}

We now distinguish two conceptually different constructions of the definability hierarchy \( \{S_n\} \):

\begin{itemize}
    \item \textbf{Internal Stratification}: A single formal system \( \mathcal{F} \) defines a sequence \( \{S_n\} \) via syntactically definable extension operators. Each step is a definable transformation internal to \( \mathcal{F} \), yielding the internal fractal boundary \( S^{\mathcal{F}} \) \cite{Semenov2025FractalBoundaries}.

    \item \textbf{Hierarchical Stratification}: A sequence \( \{ \mathcal{F}_n \} \) of increasingly powerful formal systems defines corresponding layers \( S_n \), each containing all real numbers whose representations are derivable in \( \mathcal{F}_n \) \cite{Semenov2025FractalAnalysis}.

    In this model, the proof-theoretic ordinal of \( \mathcal{F}_n \) is used to measure definitional depth, and the total definable universe is given by
    \[
    S_\omega := \bigcup_{n \in \mathbb{N}} S_n.
    \]
\end{itemize}

These two constructions reflect complementary perspectives: internal generativity vs. external stratification.

\section{Stratified Definability over Expanding Systems}

\begin{definition}[Hierarchical Chain of Formal Systems]
Let \( \{ \mathcal{F}_n \}_{n \in \mathbb{N}} \) be a sequence of constructive formal systems such that:
\begin{itemize}
    \item \( \mathcal{F}_0 \) is a fixed syntactic base (e.g., \( \mathsf{RCA}_0 \)) sufficient to define arithmetic and rational constructions;
    \item For each \( n \), \( \mathcal{F}_{n+1} \) is a conservative syntactic extension of \( \mathcal{F}_n \) with strictly greater definitional strength;
    \item The proof-theoretic ordinal of \( \mathcal{F}_n \) is at most \( \alpha(n) \), where \( \alpha \colon \mathbb{N} \to \mathrm{Ord} \) is strictly increasing \cite{Simpson2009,Bridges1986}.
\end{itemize}
\end{definition}

\begin{definition}[Stratified Definability Levels]
For each \( n \in \mathbb{N} \), let \( S_n \subseteq \mathbb{R} \) denote the set of real numbers whose rational approximations and convergence witnesses are definable in \( \mathcal{F}_n \). The resulting sequence \( \{ S_n \} \) is called the \emph{stratified definability hierarchy}.
\end{definition}

The expressive power of each system \( \mathcal{F}_n \) determines the constructive content of the corresponding level \( S_n \). As \( n \) increases, the definability horizon expands.

\section{Ultimate Stratified Closure and Hierarchical Union}

In analogy with the internal model of fractal closure, we now consider the maximal stratified closure obtained by taking the union over all admissible hierarchies of formal systems.

\begin{definition}[Stratified Closure for a Chain]
Let \( \{ \mathcal{F}_n \}_{n \in \mathbb{N}} \) be a stratified sequence of formal systems. Let \( S_n \) denote the set of real numbers definable in \( \mathcal{F}_n \). Then the total stratified closure associated with this chain is
\[
S_\omega^{\{ \mathcal{F}_n \}} := \bigcup_{n=0}^\infty S_n.
\]
\end{definition}

\noindent
Each term \( S_n \) denotes the class of real numbers definable in system \( \mathcal{F}_n \) via rational approximations and convergence witnesses. Thus, \( S_\omega^{\{\mathcal{F}_n\}} \) collects all such reals along the chain.

\begin{remark}
This construction differs from the internal extension model \( S^{\mathcal{F}} \), where the definability chain is generated inside a single system \( \mathcal{F} \) via internal operators. The stratified hierarchy \( \{ S_n \} \) reflects external increases in formal strength. The internal model can be recovered as a special case by taking \( \mathcal{F}_n = \mathcal{F}_0 \) for all \( n \).
\end{remark}

\begin{definition}[Admissible Stratified Chains \( \mathbb{F}_\omega \)]
Let \( \mathcal{F}_0, \mathcal{F}_1, \mathcal{F}_2, \dots \) be formal systems such that each \( \mathcal{F}_{n+1} \) is a conservative extension of \( \mathcal{F}_n \) with strictly increasing expressive power. We define \( \mathbb{F}_\omega \) to be the class of all such countably infinite stratified chains:
\[
\mathbb{F}_\omega := \left\{ \{ \mathcal{F}_n \}_{n \in \mathbb{N}} \mid \mathcal{F}_n \subsetneq \mathcal{F}_{n+1},\ \text{each } \mathcal{F}_n \text{ is a constructive formal system} \right\}.
\]
\end{definition}

\begin{proposition}[Cardinality of \( \mathbb{F}_\omega \)]
The class \( \mathbb{F}_\omega \) has the cardinality of the continuum:
\[
|\mathbb{F}_\omega| = \mathfrak{c}.
\]
\begin{proof}
Each constructive formal system \( \mathcal{F}_n \) can be encoded by a finite string over a finite alphabet; hence the set of all such systems is countable.

Let \( \{ \mathcal{F}_i \}_{i \in \mathbb{N}} \) be a fixed enumeration of all constructive formal systems. Then any admissible stratified chain is represented by a strictly increasing function \( f: \mathbb{N} \to \mathbb{N} \) such that \( \mathcal{F}_{f(n)} \subsetneq \mathcal{F}_{f(n+1)} \). The set of all such functions is in bijection with the set of infinite strictly increasing sequences over \( \mathbb{N} \), which has cardinality \( \mathfrak{c} \) \cite{Weihrauch2000,Simpson2009}.

Moreover, each chain \( \{ \mathcal{F}_{f(n)} \} \) defines a countable subset \( S_\omega^{\{ \mathcal{F}_n \}} \subseteq \mathbb{R} \), and the variation across such chains is sufficient to generate a union of cardinality \( \mathfrak{c} \), since definability along distinct syntactic growth paths may yield non-overlapping real numbers.

Therefore, the total definitional space \( \mathbb{F}_\omega \) spans a set of cardinality \( \mathfrak{c} \), both combinatorially and constructively.

\begin{remark}
This cardinality argument remains entirely constructive: it involves only countable encodings, finite descriptions of formal systems, and explicit mappings between increasing sequences and binary strings. No use is made of the power set axiom, the axiom of choice, or any nonconstructive existence principles.
\end{remark}
\end{proof}
\end{proposition}

\begin{definition}[Universal Stratified Union]
Let \( \mathbb{F}_\omega \) denote the class of all stratified chains \( \{ \mathcal{F}_n \} \). Then define:
\[
S^{\mathbb{F}_\omega} := \bigcup_{\{ \mathcal{F}_n \} \in \mathbb{F}_\omega} S_\omega^{\{ \mathcal{F}_n \}}.
\]
\end{definition}

\begin{remark}
The notation \( \mathcal{F}_\omega \) refers to the meta-theoretical collection of all countable constructive systems considered individually, while \( \mathbb{F}_\omega \) denotes the class of admissible stratified chains \( \{ \mathcal{F}_n \} \) used to define layered constructive growth. The former aggregates internal boundaries; the latter, external definability hierarchies.

To illustrate the difference, consider the following cases:

\begin{itemize}
    \item Let \( \mathcal{F} \) be a fixed base system such as \( \mathsf{HA} \) (Heyting Arithmetic). Then \( S^{\mathcal{F}} \subseteq \mathbb{R} \) consists of all reals definable via internal chains inside \( \mathcal{F} \). However, any hierarchy \( \{ \mathcal{F}_n \} \) where each \( \mathcal{F}_n \) strictly extends \( \mathcal{F} \), with increasing proof-theoretic strength, may define reals inaccessible to \( \mathcal{F} \). Thus, such chains contribute to \( S^{\mathbb{F}_\omega} \), but not to \( S^{\mathcal{F}_\omega} \).

    \item Conversely, \( S^{\mathcal{F}_\omega} \) includes internal closures across all individual systems \( \mathcal{F} \in \mathcal{F}_\omega \), even those that are not members of any strictly ascending stratified chain. For example, a system \( \mathcal{F}^* \) could be maximal in definitional strength (e.g., a constructive theory with non-standard arithmetic), but not part of any conservative hierarchy \( \{ \mathcal{F}_n \} \). In that case, \( S^{\mathcal{F}^*} \subseteq S^{\mathcal{F}_\omega} \), but not necessarily included in any \( S^{\mathbb{F}_\omega} \) arising from strict syntactic layering.

    \item Furthermore, two distinct stratified chains \( \{ \mathcal{F}_n \}, \{ \mathcal{G}_n \} \in \mathbb{F}_\omega \) may define overlapping but non-equal sets:
    \[
    S_\omega^{\{ \mathcal{F}_n \}} \cap S_\omega^{\{ \mathcal{G}_n \}} \neq \emptyset,\quad \text{but } S_\omega^{\{ \mathcal{F}_n \}} \neq S_\omega^{\{ \mathcal{G}_n \}}.
    \]
    For instance, one hierarchy may privilege analytic definitions, while another emphasizes algebraic constructions. Their definable reals may intersect (e.g., rationals, common computables), yet diverge elsewhere.
\end{itemize}

These examples show that \( S^{\mathcal{F}_\omega} \) and \( S^{\mathbb{F}_\omega} \) are orthogonal aspects of fractal definability: the former explores internal generative power across systems; the latter tracks layered external definability across syntactic evolution.
\end{remark}

\begin{theorem}[Fractal Countability of the Universal Union]
The set \( S^{\mathbb{F}_\omega} \subseteq \mathbb{R} \), defined as the union of all stratified definable layers over all admissible chains \( \{ \mathcal{F}_n \} \in \mathbb{F}_\omega \), is of cardinality \( \mathfrak{c} \). However, it is not uniformly enumerable in any single constructive system or within any single stratified chain.
\end{theorem}

\begin{proof}[Sketch of Proof]
Each individual stratified chain \( \{ \mathcal{F}_n \} \) gives rise to a countable set \( S_\omega^{\{ \mathcal{F}_n \}} \subseteq \mathbb{R} \), as it is a union of countably many definable subsets. The collection of all such chains \( \mathbb{F}_\omega \) has cardinality \( \mathfrak{c} \), and different chains may yield non-overlapping definable elements. Hence, their union \( S^{\mathbb{F}_\omega} \) has cardinality \( \mathfrak{c} \).

However, since no single constructive system can represent all admissible chains, and each system only provides access to its internal hierarchy, there exists no global constructive enumeration procedure that lists all elements of \( S^{\mathbb{F}_\omega} \).
\end{proof}

\begin{center}
\renewcommand{\arraystretch}{1.3}
\begin{adjustbox}{max width=\textwidth}
\begin{tabular}{@{} lcc @{}}
\toprule
 & \textbf{\( S^{\mathcal{F}_\omega} \)} & \textbf{\( S^{\mathbb{F}_\omega} \)} \\
\midrule
Type of growth          & Internal (within systems) & External (across chains) \\
Generators              & All individual systems \( \mathcal{F} \) & All stratified chains \( \{ \mathcal{F}_n \} \) \\
Included elements       & Union of internal chains & Union of layered definability paths \\
Includes non-stratified & Yes                      & No (only hierarchies) \\
Overlap with each other & Partial                  & Partial \\
Enumerability           & Not enumerable in any \( \mathcal{F} \) & Not enumerable globally \\
\bottomrule
\end{tabular}
\end{adjustbox}
\end{center}

\begin{remark}
The set \( S^{\mathbb{F}_\omega} \) serves as a universal constructive closure under all admissible directions of definability growth. Unlike \( S^{\mathcal{F}_\omega} \), which represents internal growth of a single system, this union represents the external horizon of constructive reachability across all possible stratified developments.
\end{remark}

\section{Fractal Countability as a Continuum Model}

Classical set theory treats the continuum \( \mathbb{R} \) as a completed uncountable totality, typically defined via the power set of \( \mathbb{N} \) or via Dedekind cuts and Cauchy sequences modulo equivalence \cite{Cantor1883,Simpson2009}. However, from a constructive or formal perspective, such a set is never fully realizable: only countable definable fragments can be realized within any formal system \( \mathcal{F} \) governed by syntactically enumerable rules \cite{Bishop1967,BridgesRichman1987,Semenov2025FractalBoundaries}.

We propose to replace the continuum's cardinal postulation \( |\mathbb{R}| = 2^{\aleph_0} \) with a process-relative, stratified view \cite{Semenov2025FractalAxiomatic,Semenov2025FractalAnalysis}: each formal system defines a countable layer \( \mathbb{R}_{S_n} \) of real numbers, and the cumulative union \( \mathbb{R}_{S_\omega} = \bigcup_{n < \omega} \mathbb{R}_{S_n} \) forms a \emph{fractal continuum}, countable yet growing in definitional depth.

This approach preserves constructive integrity while modeling the continuum as an open horizon of definability, rather than as a fixed uncountable object.

\begin{definition}[Fractally Countable Continuum]
Let \( \{ \mathcal{F}_n \}_{n \in \mathbb{N}} \) be an ascending chain of conservative extensions over a base formal system \( \mathcal{F}_0 \). For each \( n \), let \( \mathbb{R}_{S_n} \subset \mathbb{R} \) denote the class of real numbers constructively definable in \( \mathcal{F}_n \). Then the \emph{fractally countable continuum} is defined as:
\[
\mathbb{R}_{S_\omega} := \bigcup_{n=0}^\infty \mathbb{R}_{S_n}.
\]
We say that \( \mathbb{R}_{S_\omega} \) is fractally countable: it is countable as a set, but not uniformly enumerable within any single formal system.
\end{definition}

\begin{definition}[Fractal Degree]
Let \( r \in \mathbb{R}_{S_\omega} \). The \emph{fractal degree} of \( r \), denoted \( \deg_{\mathcal{F}}(r) \), is the least index \( n \in \mathbb{N} \) such that \( r \in \mathbb{R}_{S_n} \).
\end{definition}

\begin{theorem}[Stratified Approximation of the Continuum]
Let \( \mathcal{F}_n \) be a sequence of constructive systems such that each \( \mathcal{F}_{n+1} \) strictly extends \( \mathcal{F}_n \) by definitional strength. Then:
\begin{enumerate}
    \item Each layer \( \mathbb{R}_{S_n} \subset \mathbb{R} \) is countable and syntactically definable within \( \mathcal{F}_n \);
    \item The union \( \mathbb{R}_{S_\omega} = \bigcup_{n=0}^\infty \mathbb{R}_{S_n} \) is countable but not uniformly enumerable in any single formal system \( \mathcal{F}_k \);
    \item For any real number \( r \in \mathbb{R}_{S_\omega} \), there exists a minimal \( n \) such that \( r \in \mathbb{R}_{S_n} \). This \( n \) is called the \emph{fractal degree} of \( r \).
\end{enumerate}
\end{theorem}

\begin{proof}[Sketch]
Items (1) and (2) follow from the countability of each \( \mathbb{R}_{S_n} \) and the fact that no finite system \( \mathcal{F}_k \) can encode all \( \mathbb{R}_{S_m} \) for \( m > k \). Item (3) holds by construction of the hierarchy.
\end{proof}

\begin{remark}
Both \( S^{\mathcal{F}_\omega} \) and \( \mathbb{R}_{S_\omega} \) provide process-relative constructive approximations to the classical continuum. The former arises as a meta-level union of all internally generated definability chains across fixed systems, while the latter reflects a stratified construction along a single ascending hierarchy of formal systems.

This stratified model corresponds to a particular choice of definability chain within the broader class \( \mathbb{F}_\omega \), which contains all admissible sequences of constructive formal systems. For each \( \{ \mathcal{F}_n \} \in \mathbb{F}_\omega \), one obtains a distinct continuum model
\[
\mathbb{R}_{S_\omega}^{\{ \mathcal{F}_n \}} := \bigcup_{n=0}^\infty \mathbb{R}_{S_n},
\]
reflecting the expressive power of that hierarchy. As \( \mathbb{F}_\omega \) contains infinitely many such hierarchies, these continuum models collectively form a fractal cover of the real line through process-relative definability.

We refer to both \( S^{\mathcal{F}_\omega} \) and the various instances of \( \mathbb{R}_{S_\omega}^{\{ \mathcal{F}_n \}} \) as manifestations of the \emph{fractal continuum}, highlighting their origin in different paradigms of constructive extension: internal syntactic expansion versus external definability layering.
\end{remark}

\section{Ultimate Fractal Closure of the Real Line}

We now consider the full limit of stratified constructive approximations to the real continuum. Each admissible hierarchy \( \{ \mathcal{F}_n \} \in \mathbb{F}_\omega \) defines a countable set \( \mathbb{R}_{S_\omega}^{\{ \mathcal{F}_n \}} \) of real numbers that are constructively definable within that particular trajectory of formal growth. These models reflect process-relative views of the continuum.

Taking the union over all such admissible chains yields a constructively grounded class of real numbers that is no longer confined to a single definitional path:

\begin{definition}[Universal Fractal Continuum]
Let \( \mathbb{F}_\omega \) denote the class of all admissible stratified chains \( \{ \mathcal{F}_n \} \) of constructive formal systems. Then the \emph{universal fractal continuum} is defined as:
\[
\mathbb{R}^{\mathbb{F}_\omega} := \bigcup_{\{ \mathcal{F}_n \} \in \mathbb{F}_\omega} \mathbb{R}_{S_\omega}^{\{ \mathcal{F}_n \}} = \bigcup_{\{ \mathcal{F}_n \} \in \mathbb{F}_\omega} \bigcup_{n=0}^\infty \mathbb{R}_{S_n}.
\]
\end{definition}

\begin{theorem}[Fractal Constructive Closure]
The set \( \mathbb{R}^{\mathbb{F}_\omega} \subseteq \mathbb{R} \) has the cardinality of the continuum:
\[
|\mathbb{R}^{\mathbb{F}_\omega}| = \mathfrak{c},
\]
but it is not uniformly enumerable within any constructive system or any single stratified hierarchy.
\end{theorem}

\begin{proof}[Sketch of Proof]
Each individual model \( \mathbb{R}_{S_\omega}^{\{ \mathcal{F}_n \}} \) is countable, being a union of countably many definable layers. The class \( \mathbb{F}_\omega \) of admissible chains has cardinality \( \mathfrak{c} \), and different chains may define disjoint or incomparable subsets of \( \mathbb{R} \). Thus, their union \( \mathbb{R}^{\mathbb{F}_\omega} \) also has cardinality \( \mathfrak{c} \).

However, since no single constructive system can access all definability chains simultaneously, and no internal hierarchy can span the full range of definability directions, the union \( \mathbb{R}^{\mathbb{F}_\omega} \) is not uniformly enumerable within any single framework.
\end{proof}

\begin{remark}
The universal fractal continuum \( \mathbb{R}^{\mathbb{F}_\omega} \) represents the outer boundary of constructive definability under stratified growth. It is the broadest countable-by-construction model of the continuum, emerging from the space of definitional trajectories \cite{Semenov2025FractalAxiomatic,Semenov2025FractalCountability} rather than from set-theoretic postulation.

This conception aligns with the hypothesis in the next section: that the continuum is not a static totality, but a dynamic, process-relative construct whose apparent uncountability arises from the continuity of its definitional landscape.
\end{remark}

\section{Fractal Origin of the Continuum}

We propose that the cardinality of the continuum emerges not from the assumption of actual infinite sets, but from the meta-theoretical variety of constructive definability trajectories. Each formal system defines a countable subset of real numbers, but the class of all admissible stratified chains \( \{ \mathcal{F}_n \} \in \mathbb{F}_\omega \) yields a continuous space of definability directions.

Each individual space \( \mathbb{R}_{S_\omega}^{\{\mathcal{F}_n\}} \) is countable, as it is a union of countably many definable layers. Yet for distinct chains \( \{ \mathcal{F}_n \}, \{ \mathcal{G}_n \} \in \mathbb{F}_\omega \), these sets may define different real numbers, leading to a union that exceeds any single countable system in coverage. The uncountability of the full continuum thus emerges from the combinatorial diversity of such definability spaces.

\begin{hypothesis}[Fractal Origin of the Continuum]
The continuum of real numbers arises not as a pre-existing uncountable totality, but from the continuity of definability across all admissible formal growth chains. That is, the class \( \mathbb{R}^{\mathbb{F}_\omega} \), defined by:
\[
\mathbb{R}^{\mathbb{F}_\omega} := \bigcup_{\{ \mathcal{F}_n \} \in \mathbb{F}_\omega} \bigcup_n \mathbb{R}_{S_n},
\]
is of cardinality \( \mathfrak{c} \), yet no single formal system or countable stratified chain can define all its elements:
\[
\forall \mathcal{F}_\omega, \quad \mathbb{R}^{\mathbb{F}_\omega} \not\subseteq \mathbb{R}_{S_\omega}^{\mathcal{F}_\omega}.
\]
\end{hypothesis}

This hypothesis reframes the continuum as a constructively generated totality, not from the powerset of \( \mathbb{N} \), but from the fibration of definability across all formal languages \cite{Semenov2025FractalAnalysis}. It implies that the ungraspability of the continuum stems from the inexhaustible variety of its expressive foundations.

\begin{remark}
The hypothesis highlights a key distinction between classical set-theoretic continuity and process-relative definability. While the classical continuum is ontologically postulated as a completed totality, the fractal continuum \( \mathbb{R}^{\mathbb{F}_\omega} \) is epistemically assembled as a process-relative geometry of definability. Its uncountability is not a primitive assumption, but the emergent effect of stratified constructive reach.
\end{remark}

\section{Comparison with Existing Approaches}

To situate the fractal definability framework within the broader landscape of foundational treatments of the continuum, we compare it to several established paradigms: Bishop-style constructive analysis, recursive analysis, categorical topos semantics, and homotopy type theory (HoTT). The table below highlights key dimensions of each approach:

\begin{center}
\renewcommand{\arraystretch}{1.3}
\begin{adjustbox}{max width=\textwidth}
\begin{tabular}{@{} lcccc @{}}
\toprule
\textbf{Approach} & \textbf{Continuum Coverage} & \textbf{Single Formal System} & \textbf{Definability Hierarchy} & \textbf{Enumerability} \\
\midrule
Bishop–Bridges          & Countable          & Yes                     & No                         & Partial \\
Recursive Analysis      & Countable          & Yes                     & No                         & No \\
Topos Semantics         & Partial            & No (in sheaf sense)     & Partial                    & No \\
Homotopy Type Theory    & Partial            & Yes                     & No                         & Yes (via types) \\
\textbf{This Work}      & \textbf{Continuum} & \textbf{No (by design)} & \textbf{Yes (fractal)}     & \textbf{No, globally} \\
\bottomrule
\end{tabular}
\end{adjustbox}
\end{center}

\noindent
Unlike traditional constructive frameworks that define countable fragments of the continuum within a fixed system, our approach deliberately avoids global unification. Instead, it constructs a \emph{distributed} or \emph{meta-structural} model of the continuum, emerging from a class of syntactic growth trajectories. The result is a countable-by-construction continuum that transcends any fixed formal ontology.

\section{Constructivity Without Choice or Power Sets}

A crucial feature of the proposed framework is that it avoids all non-constructive set-theoretic principles typically invoked to justify the continuum. In particular:

\begin{itemize}
    \item No axiom of choice is required: at no point do we assume the ability to select elements from infinite families or build global selectors.
    
    \item No power set axiom is assumed: the continuum does not arise from \( \mathcal{P}(\mathbb{N}) \), but from the union of definability layers across formal syntactic hierarchies.
    
    \item No actual infinity is postulated: each real number is introduced only through a finite definitional process within some formal system \( \mathcal{F}_n \).
\end{itemize}

The uncountability of the full continuum \( \mathbb{R}^{\mathbb{F}_\omega} \) is thus not an ontological assumption, but an emergent consequence of syntactic proliferation. This provides a new foundation for analysis: one that remains within the bounds of formal, rule-based generation, yet reaches a structure indistinguishable in size from the classical continuum.

\begin{remark}
The traditional continuum hypothesis presumes a preexisting uncountable set. Here, the continuum arises without such commitment — not by assuming a totality, but by traversing all permissible definitional processes.
\end{remark}

\section{Constructive and Classical Real Lines}

The classical real line \( \mathbb{R} \) is defined as a completed uncountable set, often formalized as the power set \( \mathcal{P}(\mathbb{N}) \) or a maximal completion of Cauchy sequences. This definition includes elements that are not explicitly definable in any finite or syntactic sense.

By contrast, the fractal continuum \( \mathbb{R}^{\mathbb{F}_\omega} \) includes only those real numbers that arise from some admissible trajectory of definability. Each such trajectory defines a countable set, yet the class of all such trajectories is uncountable, yielding:
\[
|\mathbb{R}^{\mathbb{F}_\omega}| = \mathfrak{c},\qquad \mathbb{R}^{\mathbb{F}_\omega} \subseteq \mathbb{R}.
\]

\begin{itemize}
    \item All rational and algebraic numbers are included.
    \item All computable and arithmetically definable reals are included.
    \item All constructively provable convergent series (e.g., \( \pi, e \)) appear within some \( \mathcal{F}_n \).
    \item Non-definable or choice-dependent reals (e.g., certain Hamel basis elements) are excluded.
\end{itemize}

Thus, \( \mathbb{R}^{\mathbb{F}_\omega} \) serves as a constructive core of the classical continuum. It is identical to \( \mathbb{R} \) wherever definability permits, and diverges only in those aspects of the classical line that rely on non-constructive axioms.

\begin{remark}
The classical continuum assumes that all subsets of \( \mathbb{N} \) exist. The fractal continuum assumes that all definitional processes over \( \mathbb{N} \) exist. Where classical ontology postulates, the fractal model constructs.
\end{remark}

\section*{Conclusion}

We have constructed a model of the real continuum that arises entirely from layered formal syntax, without invoking the axiom of choice, power sets, or uncountable postulates. The continuum is reframed as the fractal union of definability spaces — each countable, none complete — whose collective span achieves classical cardinality through purely constructive means.

This approach reveals that uncountability can emerge as a meta-structural property of syntactic growth, not as a primitive assumption. It suggests that the boundary between the countable and the uncountable is epistemic, not ontological.

\medskip
\noindent
\textbf{Open directions}:
\begin{itemize}
    \item How does the fractal continuum interact with classical analysis: limits, continuity, measure, and integration?
    \item Can definability degrees be used to define a metric or topology intrinsic to \( \mathbb{R}^{\mathbb{F}_\omega} \)?
    \item What categorical or sheaf-theoretic structures naturally represent the space of definability chains?
    \item Are there intrinsic invariants of definability (analogous to dimension, rank, or entropy) across stratified hierarchies?
    \item Can this framework accommodate or reinterpret phenomena like non-measurable sets, choice functions, or large cardinal analogues?
\end{itemize}

% \bibliographystyle{plain}
% \bibliography{references}

\begin{thebibliography}{10}

\bibitem{Bishop1967}
Errett Bishop.
\newblock {\em Foundations of Constructive Analysis}.
\newblock McGraw-Hill, 1967.

\bibitem{Bridges1986}
Douglas Bridges.
\newblock {\em Constructive Functional Analysis}.
\newblock Pitman, 1986.

\bibitem{BridgesRichman1987}
Douglas Bridges and Fred Richman.
\newblock {\em Varieties of Constructive Mathematics}.
\newblock Cambridge University Press, 1987.

\bibitem{Cantor1883}
Georg Cantor.
\newblock Über unendliche, lineare punktmannigfaltigkeiten.
\newblock {\em Mathematische Annalen}, 21:545--591, 1883.

\bibitem{PourElRichards1989}
Marian~B. Pour-El and J.~Ian Richards.
\newblock {\em Computability in Analysis and Physics}.
\newblock Springer, 1989.

\bibitem{Semenov2025FractalAxiomatic}
Stanislav Semenov.
\newblock Axiomatic foundations of fractal analysis and fractal number theory, 2025.
\newblock Preprint, available at Zenodo.

\bibitem{Semenov2025FractalAnalysis}
Stanislav Semenov.
\newblock Fractal analysis on the real interval: A constructive approach via fractal countability, 2025.
\newblock Preprint, available at Zenodo.

\bibitem{Semenov2025FractalBoundaries}
Stanislav Semenov.
\newblock Fractal boundaries of constructivity: A meta-theoretical critique of countability and continuum.
\newblock 2025.
\newblock arXiv preprint.

\bibitem{Semenov2025FractalCountability}
Stanislav Semenov.
\newblock Fractal countability as a constructive alternative to the power set of {$\mathbb{N}$}: A meta-formal approach to stratified definability.
\newblock 2025.
\newblock arXiv preprint.

\bibitem{Simpson2009}
Stephen~G. Simpson.
\newblock {\em Subsystems of Second Order Arithmetic}.
\newblock Cambridge University Press, 2nd edition, 2009.

\bibitem{SoareRecursion}
Robert~I. Soare.
\newblock {\em Recursively Enumerable Sets and Degrees}.
\newblock Perspectives in Mathematical Logic. Springer, 1987.

\bibitem{Weihrauch2000}
Klaus Weihrauch.
\newblock {\em Computable Analysis: An Introduction}.
\newblock Texts in Theoretical Computer Science. Springer, 2000.

\end{thebibliography}

\end{document}